\documentclass{article}
\usepackage[utf8]{inputenc}
\usepackage[colorlinks=true,linktocpage=true]{hyperref}
\usepackage{amsmath}
\usepackage{amssymb}
\usepackage{amsthm}
\usepackage{enumerate}

\numberwithin{equation}{section}

\newtheorem{thm}{Theorem}[section]
\newtheorem{ptn}[thm]{Proposition}
\newtheorem{clm}[thm]{Claim}
\newtheorem{cor}[thm]{Corollary}
\newtheorem{lem}[thm]{Lemma}
\newtheorem{con}[thm]{Conjecture}

\theoremstyle{definition}
\newtheorem{dfn}[thm]{Definition}
\newtheorem{exm}[thm]{Example}

\theoremstyle{remark}
\newtheorem{rmk}[thm]{Remark}

\newcommand{\wt}{\widetilde}
\newcommand{\ol}{\overline}
\newcommand{\C}{\mathbb{C}}
\newcommand{\calH}{\mathcal{H}}
\newcommand{\N}{\mathbb{N}}
\newcommand{\PP}{\mathbb{P}}
\newcommand{\Q}{\mathbb{Q}}
\newcommand{\R}{\mathbb{R}}
\newcommand{\Ht}[1]{\ol{#1}^\mathsf{T}}

 \DeclareMathOperator{\Int}{int}
\DeclareMathOperator{\diag}{diag} \DeclareMathOperator{\re}{Re}
\DeclareMathOperator{\im}{Im} \DeclareMathOperator{\Aut}{Aut}

\def\K{K\"ahler }

\title{Chebyshev potentials, Fubini--Study metrics, and geometry of the space of K\"ahler metrics\thanks{Research supported
by NSF grants DMS-1906370,2204347.}}
\author{Chenzi Jin, Yanir A. Rubinstein }
\date{24 October 2022}

\makeatletter
\def\thanks#1{\protected@xdef\@thanks{\@thanks\protect\footnotetext{#1}}}
\makeatother

\begin{document}

\maketitle

\begin{abstract}
The Chebyshev potential of a K\"ahler potential on a projective
variety, introduced by Witt Nystr\"om, is a convex function defined
on the Okounkov body. It is a generalization of the symplectic
potential of a torus-invariant \K potential on a toric variety,
introduced by Guillemin, that is a convex function on the Delzant
polytope. A folklore conjecture asserts that a curve of Chebyshev
potentials associated to a curve in the space of \K potentials is
linear in the time variable if and only if the latter curve is a
geodesic in the Mabuchi metric. This is classically true in the
special toric setting, and in general Witt Nystr\"om established the
sufficiency. The goal of this article is to disprove this
conjecture. More generally, we characterize the Fubini--Study
geodesics for which the conjecture is true on projective space. The
proof involves explicitly solving the Monge--Amp\`ere equation
describing geodesics on the subspace of Fubini--Study metrics and
computing their Chebyshev potentials.

\end{abstract}


\section{A folklore conjecture}

Consider first the simpler setting of a toric manifold $X$. Let
$\Delta\subseteq\R_{\geqslant0}^n$ be the moment polytope associated
to a line bundle $L_\Delta\to X$. On the torus $(\C^*)^n\cong
X_\mathrm{open}\subseteq X$ define
$$
\Theta:z\mapsto\left(\log\left|z_1\right|,\ldots,\log\left|z_n\right|\right).
$$
Of course, $\Theta$ is not invertible, but if a positively curved
metric $e^{-\varphi}$ on $L_\Delta$ (i.e., $\varphi$ is
plurisubharmonic (psh)) is torus-invariant on $X_\mathrm{open}$, the
function $\varphi_\Theta=\varphi\circ\Theta^{-1}$ is well-defined
and convex on $\R^n$.

The geodesic equation on the space of Hermitian metrics on $L$ with
positive curvature endowed with the Mabuchi $L^2$ metric becomes, as
observed by Semmes, and later Donaldson, the homogeneous complex
Monge--Amp\`ere equation. This is an equation for a complex curve
$t+\sqrt{-1} s\mapsto \varphi(t)$ with $\varphi$ now psh in all
$n+1$ variables. Under the toric symmetry assumption, the equation
simplifies to the homogeneous real Monge--Amp\`ere equation (HRMA)
$$\begin{aligned}
\det\left(\nabla_{t,x}^2\varphi_\Theta\right)&=0,&\mbox{on
}\left[0,T\right]\times\R^n
\end{aligned}$$
for a convex function in all $n+1$ variables. For each fixed time
$t$, the partial Legendre transform $\mathcal{L}$ maps
$\varphi_\Theta(t,\,\cdot\,)$ to a convex function on $\Delta$.
Under this correspondence, the geodesic equation further reduces to
$\frac{d^2}{dt^2}\mathcal{L}\varphi_\Theta(t,\,\cdot\,)=0$. That is,
$t\mapsto \varphi(t)$
 is a geodesic precisely when $t\mapsto\mathcal{L}\varphi_\Theta(t)$ is affine in $t$. These facts go back to Mabuchi, Semmes,  and
Donaldson  \cite[Section 1]{Sem}, \cite[Section 3.1]{Mab},
\cite[Section 6]{Don}.

From now on, let $X$ be a projective manifold and $L\rightarrow X$ a
line bundle with $c_1(L)$ a K\"ahler class. Denote by $$ \calH_L
$$ the space
of K\"ahler potentials whose associated \K forms represent $c_1(L)$.
The Chebyshev potential $c[\varphi]$ of $\varphi$, introduced by
Witt Nystr\"om \cite[Section 5]{Wit}, is a generalization of the
symplectic potential, introduced by Guillemin \cite[Section 4]{Gui},
to any projective, but not necessarily toric, manifold. It is
defined on the interior of the Okounkov body
$\Delta(X,L,Y_{\bullet})$ of the line bundle $L$, itself a
generalization of the (Delzant) moment polytope $\Delta$---see
Definition \ref{Chebydef}.

In the special case of toric $X$, the symplectic potential and the
Chebyshev potential are related by the formula
$
c\left[\varphi\right]=2\mathcal{L}\left(\frac{1}{2}\varphi_\Theta\right)
$ \cite[Section 10.3]{Wit}.
%
Therefore, it is natural to pose the following conjecture, first
explicitly stated by Reboulet  \cite[Theorem A]{Reb}, that, by the
above discussion, holds in the toric case.
\begin{con}\label{conj}
The segment $t\mapsto\varphi(t,\,\cdot\,)\in \calH_L $ is a geodesic
if and only if $t\mapsto c[\varphi(t,\,\cdot\,)]$ is affine.
\end{con}

In fact, one direction of Conjecture \ref{conj} is known. The main
result of Witt Nystr\"om \cite[Theorem 6.2]{Wit} shows that the
Aubin--Mabuchi energy $\mathcal{E}:\calH_L\times\calH_L\to\R$,
$$
\mathcal{E}\left(\varphi_0,\varphi\right):=\frac{1}{n+1}\sum_{j=0}^n\int_X\left(\varphi_0-\varphi\right)\left(dd^c\varphi_0\right)^j\wedge\left(dd^c\varphi\right)^{n-j},
$$
defined in terms of \K potentials on $X$ is already encoded in the
Chebyshev potentials and the Okounkov body,
$$
\mathcal{E}\left(\varphi_0,\varphi\right)=n!\int_{\Delta\left(L\right)}\left(c\left[\varphi_0\right]-c\left[\varphi\right]\right)d\mu,
$$
where $d\mu$ is the standard Lebesgue measure.
As first observed by Mabuchi in the smooth setting (and later
generalized by Berman--Boucksom \cite[(4.1)]{BBInv}),
$t\mapsto\varphi(t,\,\cdot\,)$ is a geodesic if and only if
$t\mapsto\mathcal{E}\left(\varphi\left(t,\,\cdot\,\right),\varphi\left(0,\,\cdot\,\right)\right)$
is affine \cite[Remark 3.3]{Mab}. This implies that
$\varphi(t,\,\cdot\,)$ is a geodesic if and only if the integral $
\int_{\Delta\left(L\right)}c\left[\varphi\left(t,\,\cdot\,\right)\right]
$ is affine in $t$, proving one direction
of Conjecture \ref{conj}.


Our main result is an explicit computation of the Chebyshev
potential of Fubini--Study metrics. This seems to the first explicit
computation of the Chebyshev potential outside of toric metrics and
the Riemann sphere \cite[\S10]{Wit}.

\begin{thm}\label{CompCheby}
Let $[Z_0:\ldots:Z_n]$ be standard homogeneous coordinates on
$\PP^n$. Fix the flag $\PP^n=Y_0\supseteq\cdots\supseteq Y_n=\{p\}$
where $$ Y_i=V(Z_0,\ldots,Z_{i-1}).
$$
Choose the standard coordinate chart on $U_n=\{Z_n\neq0\}$ and the
standard local trivialization of the hyperplane bundle $H$ over
$U_n$. Associate a Fubini--Study potential ${\varphi_P}$ to any
positive definite Hermitian matrix $P$ of order $n+1$ given by
$$
\varphi_P:=\log\Ht{z}Pz\in\calH_L,
$$
where $z=(z_0,\ldots,z_{n-1},1)$. Then its Chebyshev potential is
given by
$$
c\left[\varphi_P\right]\left(\alpha\right)=\sum_{i=0}^{n-1}\alpha_i\log\frac{\alpha_i}{\mu_i\left(P\right)}+\left(1-\sum_{i=0}^{n-1}\alpha_i\right)\log\frac{\left(1-\sum\limits_{i=0}^{n-1}\alpha_i\right)}{\mu_i\left(P\right)},
$$
where $\alpha= (\alpha_0,\ldots,\alpha_{n-1})$ lies in the simplex
\eqref{simplexEq} $\Delta(X,H,Y_{\bullet})\subset  \R^n$, and
$\mu_i$ is defined in (\ref{defmu}).
\end{thm}

This result is motivated by Conjecture \ref{conj} but is of
independent interest. It has the following consequence.

\begin{cor}\label{introaff}
For a geodesic $t\mapsto\varphi_{P(t)}\in\calH_L$, $t\mapsto
c[\varphi_{P(t)}]$ is affine if and only if there is a lower
triangular matrix $L\in GL(n+1,\C)$ with positive diagonal entries
and a real diagonal matrix $K$ such that
$$
P\left(t\right)=\Ht{L}e^{tK}L.
$$
Such decomposition is unique.
\end{cor}


Next, we record the following, probably well-known fact, for which
we could not find a reference. It describes geodesics of
Fubini--Study potentials (which, we emphasize, is not quite the same
as the easier task of describing geodesics of Fubini--Study K\"ahler
metrics).

\begin{ptn}\label{introgeo}
The segment $t\mapsto\varphi_{P(t)}\in\calH_L$ is a geodesic if and
only if there is an $A\in GL(n+1,\C)$ and a real diagonal matrix $D$
such that
$$
P\left(t\right)=\Ht{A}e^{tD}A,
$$
where $\varphi_{P(t)}$ is defined in Section \ref{Chebycomp}.
\end{ptn}

Combining the previous two results, we obtain explicit
counterexamples to Conjecture \ref{conj}.

\begin{cor}
Consider the hyperplane bundle $H\to\PP^n$. There exists a geodesic
$t\mapsto\varphi(t,\,\cdot\,)\in\calH_H$, such that $t\mapsto
c[\varphi(t,\,\cdot\,)]$ is not affine.
\end{cor}




\begin{proof}
Let $n=1$. That is, $X=\PP^1$. Consider
$$
P\left(t\right):=\begin{pmatrix}
\cosh t&\sinh t\\
\sinh t&\cosh t
\end{pmatrix}=\begin{pmatrix}
\frac{\sqrt{2}}{2}&-\frac{\sqrt{2}}{2}\\
\frac{\sqrt{2}}{2}&\frac{\sqrt{2}}{2}
\end{pmatrix}\begin{pmatrix}
e^t\\
&e^{-t}
\end{pmatrix}\begin{pmatrix}
\frac{\sqrt{2}}{2}&\frac{\sqrt{2}}{2}\\
-\frac{\sqrt{2}}{2}&\frac{\sqrt{2}}{2}
\end{pmatrix}.
$$
By Proposition \ref{introgeo}, $\varphi_{P(t)}$ is a geodesic.
However, by (\ref{defmu}),
$$\begin{aligned}
\mu_0\left(P\left(t\right)\right)&=\frac{1}{\cosh
t},&\mu_1\left(P\left(t\right)\right)&=\cosh t.
\end{aligned}$$
Therefore by Theorem \ref{CompCheby}, $t\mapsto c[\varphi_{P(t)}]$
is not affine.
\end{proof}

\subsection*{Organization}

In Section \ref{PreCheby}, we recall definitions of the Okounkov
body and the Chebyshev potential. As an application, in Section
\ref{Chebycomp} we compute the Chebyshev potentials of Fubini--Study
metrics and prove Theorem \ref{CompCheby}. In Section \ref{Secgeod},
we give the definition of a geodesic. Then we use quantization to
find geodesics in the space of Fubini--Study metrics and prove
Proposition \ref{introgeo}. In Section \ref{Seccount}, we classify
all geodesics whose Chebyshev potentials are affine in $t$, and
prove Corollary \ref{introaff}.

\section{Chebyshev potentials of Fubini--Study potentials}

\subsection{Background on Chebyshev potentials}\label{PreCheby}


This section gives some background on Chebyshev potentials.

First recall the definition of the Okounkov body of a big line
bundle with respect to an admissible flag \cite[Section 1]{LM}.


\begin{dfn}
Let $X^n$ be a projective complex manifold. A flag
$Y_\bullet:X=Y_0\supseteq\cdots\supseteq Y_n=\{p\}$ is
\textit{admissible} if each $Y_i$ is a submanifold of codimension
$i$.
\end{dfn}

\begin{dfn}\label{Okb}
Let $Y_\bullet:X=Y_0\supseteq\cdots\supseteq Y_n=\{p\}$ be an
admissible flag on $X$, and $L\to X$ a big line bundle. Choose a
local frame $f:U\to D^n\subseteq\C^n$ around $p$ such that
$$
f\left(Y_i\cap U\right)=\left\{z\in
D^n\,\middle|\,z_1=\cdots=z_i=0\right\},
$$
and fix a local trivialization of $L|_U$. A section $s\in H^0(X,mL)$
has a Taylor expansion
$$
s=\sum_\alpha k_\alpha z^\alpha
$$
on $U$. If $s\neq0$, define
$$
\nu\left(s\right):=\min\left\{\alpha\,\middle|\,k_\alpha\neq0\right\}\in\N^n,
$$
with respect to the lexicographic order, i.e., $\alpha<\beta$ if
$\alpha_i<\beta_i$ for some $i$ and $\alpha_j=\beta_j$ for $j<i$.
Let
$$
\Delta_m\left(L\right):=\left\{\nu\left(s\right)\,\middle|\,s\in
H^0\left(X,mL\right)\setminus\left\{0\right\}\right\}\subseteq\N^n.
$$
Define the \textit{Okounkov body} associated to $(X,Y_\bullet,L)$
$$
\Delta\left(X,Y_\bullet,L\right)=\Delta\left(L\right):=\ol{\bigcup_m\frac{1}{m}\Delta_m\left(L\right)}\subseteq\R^n.
$$
\end{dfn}

The following proposition is proved in \cite[Lemma 1.4]{LM}.
\begin{ptn}\label{dimH0}
The set $\Delta_m(L)$ has exactly $\dim H^0(X,mL)$ points.
\end{ptn}

\begin{exm}\label{exm}
Consider $X=\PP^n$ with the admissible flag
$\PP^n=Y_0\supseteq\cdots\supseteq Y_n$ where
$Y_i=V(Z_0,\ldots,Z_{i-1})$. The local frame around
$p=[0:\cdots:0:1]$ can be chosen to be the standard open set
$U_n=\{Z_n\neq0\}$ with the standard coordinate chart
$z_i=\frac{Z_i}{Z_n}$, where $0\leq i\leq n-1$. Choose the standard
local trivialization of the hyperplane bundle $H$ over $U_n$, i.e.,
a map $H|_{U_n}\to U_n\times\C$ such that for a section $s=z_i$,
$f(s(z_0,\ldots,z_{n-1}))=z_i$. The sections of the line bundle $mH$
are polynomials in $z_0,\ldots,z_{n-1}$ of degree at most $m$. Hence
$$
\Delta_m\left(H\right)=\left\{\left(\alpha_0,\ldots,\alpha_{n-1}\right)\in\N^n\,\middle|\,\sum_{i=0}^{n-1}\alpha_i\leq
m\right\},
$$
and the Okounkov body is the simplex
\begin{equation}
\label{simplexEq}
\begin{aligned}
\Delta\left(H\right)&=\ol{\left\{\left(\alpha_0,\ldots,\alpha_{n-1}\right)\in\Q_{\geqslant0}^n\,\middle|\,\sum_{i=0}^{n-1}\alpha_i\leq1\right\}}\\
&=\left\{\left(\alpha_0,\ldots,\alpha_{n-1}\right)\in\R_{\geqslant0}^n\,\middle|\,\sum_{i=0}^{n-1}\alpha_i\leq1\right\}.
\end{aligned}
\end{equation}
\end{exm}

The definition of the Chebyshev potential involves certain sections
called the Chebyshev sections, as introduced by Witt Nystr\"om
\cite[Sections 5 and 7]{Wit}.
\begin{dfn}
Notation as in Definition \ref{Okb}. The \textit{leading term} of a
section $\displaystyle s=\sum_{\alpha\geq\nu(s)}k_\alpha z^\alpha$
is defined to be
$$
\ell(s):=k_{\nu\left(s\right)}z^{\nu\left(s\right)}.
$$
For $\alpha\in\frac{1}{m}\Delta_m(L)$, define an affine space
$$
A_{m,\alpha}:=\left\{s\in
H^0\left(X,mL\right)\,\middle|\,\ell\left(s\right)=z^{m\alpha}\right\}.
$$
\end{dfn}
\begin{dfn}\label{Chebysec}
Notation as in Definition \ref{Okb}. Let $\varphi$ be a continuous
potential on $L$, an $\mu$ a smooth volume form on $X$. The linear
space $H^0(X,mL)$ admints a Hermitian inner product
$$
Hilb_m\left(\varphi,\mu\right)\left(s,t\right):=\int_X
s\ol{t}e^{-\varphi}d\mu.
$$
Each affine subspace $A_{m,\alpha}$ has a unique minimizer
$Ch_{m,\alpha}$ of the norm function \cite[Corollary 5.4]{Bre}. Such
a minimizer is called a \textit{Chebyshev section} for
$(X,Y_\bullet,L,\varphi,\mu,\alpha,m)$.
\end{dfn}

\begin{rmk}
The definition of $A_{m,\alpha}$ depends on the choice of the chart.
Indeed, monic sections are not preserved by the scaling of the
chart.
\end{rmk}

The following definition of the Chebyshev potential can be found in
\cite[Definition 5.5]{Wit}.
\begin{dfn}\label{Chebydef}
Notation as in Definition \ref{Okb}. Fix a smooth volume form $\mu$
and let $Ch_{m,\alpha}$ denote the Chebyshev sections as in
Definition \ref{Chebysec}. Given a point $\alpha\in\Int\Delta(L)$,
pick a sequence $\alpha_{m_k}\in\frac{1}{m_k}\Delta_{m_k}(L)$ that
converges to $\alpha$, where $m_k$ is any strictly increasing
sequence of natural numbers. Define the \textit{Chebyshev potential}
$$
c\left[\varphi\right]\left(\alpha\right):=\lim_{k\to\infty}\frac{1}{m_k}\log\left\|Ch_{m_k,\alpha_{m_k}}\right\|_{\varphi,\mu}^2.
$$
This limit always exists, and is independent of the choice of
$\alpha_{m_k}$ and $\mu$ \cite[Proposition 7.3]{Wit}.
\end{dfn}

The following is from \cite[Proposition 7.3]{Wit}.
\begin{ptn}\label{cvx}
The Chebyshev potential is convex on $\Delta(X,Y_\bullet,L)$.
\end{ptn}

\subsection{A Gram--Schmidt
criterion for Chebyshev sections}

\begin{ptn}\label{orthoChebysec}
For each $m$, the collection of all the Chebyshev sections
$\{Ch_{m,\alpha}\}_{\alpha\in\frac{1}{m}\Delta_m(L)}$ forms an
orthogonal basis for $H^0(X,mL)$. On the other hand, if
$\{s_\alpha\}_{\alpha\in\frac{1}{m}\Delta_m(L)}$ is an orthogonal
basis for $H^0(X,mL)$ with $\ell(s_\alpha)=z^{m\alpha}$, then
$s_\alpha$ is the Chebyshev section in $A_{m,\alpha}$.
\end{ptn}
\begin{proof}
Pick distinct $\alpha_1,\alpha_2\in\frac{1}{m}\Delta_m(L)$. One may
assume $\alpha_1<\alpha_2$ (recall Definition \ref{Okb}), so
$Ch_{m,\alpha_2}$ has higher order than $\alpha_1$, and
$$
Ch_{m,\alpha_1}-\frac{Hilb_m\left(\varphi,\mu\right)\left(Ch_{m,\alpha_1},Ch_{m,\alpha_2}\right)}{\left\|Ch_{m,\alpha_2}\right\|_{\varphi,\mu}^2}Ch_{m,\alpha_2}\in
A_{m,\alpha_1}.
$$
Since $Ch_{m,\alpha_1}$ is a Chebyshev section,
$$\begin{aligned}
\left\|Ch_{m,\alpha_1}\right\|_{\varphi,\mu}^2&\leq\left\|Ch_{m,\alpha_1}-\frac{Hilb_m\left(\varphi,\mu\right)\left(Ch_{m,\alpha_1},Ch_{m,\alpha_2}\right)}{\left\|Ch_{m,\alpha_2}\right\|_{\varphi,\mu}^2}Ch_{m,\alpha_2}\right\|_{\varphi,\mu}^2\\
&=\left\|Ch_{m,\alpha_1}\right\|_{\varphi,\mu}^2-\frac{\left|Hilb_m\left(\varphi,\mu\right)\left(Ch_{m,\alpha_1},Ch_{m,\alpha_2}\right)\right|^2}{\left\|Ch_{m,\alpha_2}\right\|_{\varphi,\mu}^2}.
\end{aligned}$$
This shows that $Ch_{m,\alpha_1}$ and $Ch_{m,\alpha_2}$ are
orthogonal. By Proposition \ref{dimH0}, it follows that
$\{Ch_{m,\alpha}\}_{\alpha\in\frac{1}{m}\Delta_m(L)}$ is an
orthogonal basis for $H^0(X,mL)$.

Let $\{s_\alpha\}_{\alpha\in\frac{1}{m}\Delta_m(L)}$ be an
orthogonal basis for $H^0(X,mL)$, where
$\ell(s_\alpha)=z^{m\alpha}$. For any $\alpha_0$, one can write
$$
Ch_{m,\alpha_0}=\sum_{\alpha\in\frac{1}{m}\Delta_m\left(L\right)}k_\alpha
s_\alpha.
$$
Define
$$
\wt{\alpha}:=\min\{\alpha\,|\,k_\alpha\neq0\}.
$$
Then
$$
z^{m\alpha_0}=\ell\left(Ch_{m,\alpha_0}\right)=\ell\left(\sum_{\alpha\in\frac{1}{m}\Delta_m\left(L\right)}k_\alpha
s_\alpha\right)=k_{\wt{\alpha}}z^{m\wt{\alpha}}.
$$
Therefore $\wt{\alpha}=\alpha_0$ and $k_{\alpha_0}=1$. That is,
\begin{equation}\label{teqs}
Ch_{m,\alpha_0}=s_{\alpha_0}+\sum_{\alpha>\alpha_0}k_\alpha
s_\alpha.
\end{equation}
Since $Ch_{m,\alpha_0}$ is a Chebyshev section,
$$
\left\|Ch_{m,\alpha_0}\right\|_{\varphi,\mu}\leq\left\|s_{\alpha_0}\right\|_{\varphi,\mu}.
$$
On the other hand, $\{s_\alpha\}$ is an orthogonal basis. Thus
$$\begin{aligned}
\left\|Ch_{m,\alpha_0}\right\|_{\varphi,\mu}^2&=\left\|s_{\alpha_0}+\sum_{\alpha>\alpha_0}k_\alpha s_\alpha\right\|_{\varphi,\mu}^2\\
&=\left\|s_{\alpha_0}\right\|_{\varphi,\mu}^2+\sum_{\alpha>\alpha_0}\left|k_\alpha\right|^2\left\|s_\alpha\right\|_{\varphi,\mu}^2\\
&\geqslant\left\|Ch_{m,\alpha_0}\right\|_{\varphi,\mu}^2+\sum_{\alpha>\alpha_0}\left|k_\alpha\right|^2\left\|s_\alpha\right\|_{\varphi,\mu}^2.
\end{aligned}$$
It follows that $k_\alpha=0$ for $\alpha>\alpha_0$ and (\ref{teqs})
becomes $Ch_{m,\alpha_0}=s_{\alpha_0}$.
\end{proof}

\subsection{Computing the Chebyshev sections and potential}\label{Chebycomp}

This section is devoted to the computation of the Chebyshev
potential of Fubini--Study metrics on $\PP^n$. The following setup
from Example \ref{exm} will be used throughout this section.

Fix a set of homogeneous coordinates
$$
[Z_0:\ldots:Z_{n}]
$$
on $\PP^n$ and an admissible flag $\PP^n=Y_0\supseteq\cdots\supseteq
Y_n=\{p\}$ where
$$Y_i:=V(Z_0,\ldots,Z_{i-1}).$$
Choose the standard coordinate chart
$$z_i:=\frac{Z_i}{Z_n}$$
valid on the open set $U_n=\{Z_n\neq0\}$ around $p$, and the
standard local trivialization of the hyperplane bundle $H$ over
$U_n$. By Example \ref{exm}, the Okounkov body is the simplex
$$
\Delta\left(H\right)=\left\{\left(\alpha_0,\ldots,\alpha_{n-1}\right)\in\left[0,+\infty\right)^n\,\middle|\,\sum_{i=0}^{n-1}\alpha_i\leq1\right\}.
$$

Define $\mathcal{P}_n(\C)$ to be the space of positive definite
Hermitian matrices of order $n$. For any
$P\in\mathcal{P}_{n+1}(\C)$, define a Fubini--Study potential
$$
\varphi_P:=\log\Ht{z}Pz,
$$
where $z=(z_0,\ldots,z_{n-1},1)^\mathsf{T}$.

\begin{dfn}\label{defetamu}
Let
$$
P=\begin{pmatrix}
p_{00}&\cdots&p_{0n}\\
\vdots&\ddots&\vdots\\
p_{n0}&\cdots&p_{nn}
\end{pmatrix}
$$
be a positive definite Hermitian matrix. For $0\leq i\leq n$ set
\begin{equation}\label{defeta}
\det\nolimits_i\left(P\right):=\det\begin{pmatrix}
p_{ii}&\cdots&p_{in}\\
\vdots&\ddots&\vdots\\
p_{ni}&\cdots&p_{nn}
\end{pmatrix}.
\end{equation}
For $i>n$ set $\det_i(P):=1$. Since $P$ is positive definite,
$\det_i(P)>0$ for any $i\geq0$ \cite[Chapter X, Theorem 3]{Gan}.
Define
\begin{equation}\label{defmu}
\mu_i\left(P\right):=\frac{\det_i\left(P\right)}{\det_{i+1}\left(P\right)}.
\end{equation}
\end{dfn}

This is an analogue of the $i$-th eigenvalue. In fact, if $P$ is
diagonal, then $\mu_i(P)=p_{ii}$. To compute $\mu_i(P)$ for general
$P$, Claims \ref{mu} and \ref{mured} will be used in Lemma
\ref{linedecomp}.

Recall the notion of a unitriangular matrix \cite[Page 30]{Rot}.
\begin{dfn}
A matrix is \textit{unitriangular} if it is triangular and its
diagonal entries are 1.
\end{dfn}
\begin{clm}\label{mu}
Let $P$ be positive definite Hermitian. For any lower triangular
matrix $L$, $\mu_i(\Ht{L}PL)=\mu_i(P)\cdot|l_{ii}|^2$. In
particular, if $L$ is
unitriangular, then $\mu_i(\Ht{L}PL)=\mu_i(P)$.
\end{clm}
\begin{proof}
To compute $\mu_i$, subdivide an $(n+1)\times(n+1)$ matrix into
submatrices of sizes $i\times i$, $i\times(n+1-i)$, $(n+1-i)\times
i$ and $(n+1-i)\times(n+1-i)$,
$$\begin{aligned}
P&=\begin{pmatrix}
P_{11}&P_{12}\\
P_{21}&P_{22}
\end{pmatrix},
&L&=\begin{pmatrix}
L_{11}&0\\
L_{21}&L_{22}
\end{pmatrix}.
\end{aligned}$$
Note that $\det(P_{22})=\det_i(P)$,
$\displaystyle\det\left(L_{22}\right)=\prod_{j=i}^nl_{jj}$. By
definition,
$$\begin{aligned}
\det\nolimits_i\left(\Ht{L}PL\right)&=\det\nolimits_i\left(\begin{pmatrix}
\Ht{L_{11}}&\Ht{L_{21}}\\
0&\Ht{L_{22}}
\end{pmatrix}\begin{pmatrix}
P_{11}&P_{12}\\
P_{21}&P_{22}
\end{pmatrix}\begin{pmatrix}
L_{11}&0\\
L_{21}&L_{22}
\end{pmatrix}\right)\\
&=\det\left(\begin{pmatrix} 0&\Ht{L_{22}}
\end{pmatrix}\begin{pmatrix}
P_{11}&P_{12}\\
P_{21}&P_{22}
\end{pmatrix}\begin{pmatrix}
0\\
L_{22}
\end{pmatrix}\right)\\
&=\det\left(\Ht{L_{22}}P_{22}L_{22}\right)\\
&=\det\nolimits_i\left(P\right)\cdot\prod_{j=i}^n\left|l_{jj}\right|^2,
\end{aligned}$$
and similarly for $i+1$. Dividing concludes the proof.
\end{proof}

\begin{clm}\label{mured}
Let $P\in\mathcal{P}_n(\C)$ and $a>0$. For $0\leq i\leq n-1$,
$$
\mu_i\begin{pmatrix}
P\\
&a
\end{pmatrix}=\mu_i\left(P\right).
$$
\end{clm}
\begin{proof}
By (\ref{defeta}), for $0\leq i\leq n$, one has
$\det_i\begin{pmatrix}
P\\
&a
\end{pmatrix}=a\cdot\det_i(P)$. Using (\ref{defmu}) we conclude.
\end{proof}


We work on the chart $\{Z_n\neq0\}$. The space $H^0(\PP^n,H)$ has a
canonical basis
$$
s_0=z_0,\ldots,s_{n-1}=z_{n-1},s_n=1.
$$
For $P\in\mathcal{P}_{n+1}(\C)$, one can write $P=\Ht{L}DL$, where
$L$ is lower unitriangular, and $D$ is real diagonal \cite[Theorem
4.1.2]{VG}. Indeed, $D=\diag\{\mu_0(P),\ldots,\mu_n(P)\}$ by Claim
\ref{mu}.

To simplify the computation, use a new homogeneous coordinate on
$\PP^n$
$$
W:=LZ
$$
The standard coordinate chart on $\{W_n\neq0\}$ is
$w_i=\frac{W_i}{W_n}$ with
\begin{equation}\label{cov2}
w:=(w_0,\ldots,w_{n-1},1)^\mathsf{T}=\frac{Z_n}{W_n}Lz.
\end{equation}
Fix a smooth volume form
$$
\mu=\left(\frac{\sqrt{-1}}{2\pi}\right)^n\frac{dw_0\wedge
d\ol{w}_0\wedge\ldots\wedge dw_{n-1}\wedge
d\ol{w}_{n-1}}{\left(\Ht{w}Dw\right)^{n+1}}.
$$

For a multi-index $I=(I_0,\ldots,I_n)$ set
$$\begin{aligned}
\left|I\right|&:=I_0+\cdots+I_n,\\
I!&:=I_0!\cdots I_n!.
\end{aligned}$$

\begin{ptn}
Define the sections of $H^0(X,H)$
$$
t_i:=s_i+\sum_{j=0}^{i-1}l_{ij}s_j,
$$
or in matrix form,
\begin{equation}\label{cov1}
t:=Ls.
\end{equation}
Then for any multi-index $I$, the section $t^I\in H^0(X,|I|H)$ is
the Chebyshev section for
$(\PP^n,Y_\bullet,H,\varphi_P,\mu,\frac{1}{|I|}(I_0,\ldots,I_{n-1}),|I|)$.
\end{ptn}
\begin{proof}
For multi-indices $I$ and $J$ with $|I|=|J|=m$, $t^I$ and $t^J$ are
sections of $mH$. Note that by (\ref{cov2}) and (\ref{cov1}),
\begin{equation}\label{cov3}
e^{-\varphi_P}=\Ht{Lz}DLz=\left|\frac{W_n}{Z_n}\right|^2\Ht{w}Dw,
\end{equation}
and
\begin{equation}\label{cov4}
t=Lz=\frac{W_n}{Z_n}w.
\end{equation}
So
\begin{multline*}
Hilb_m\left(\varphi_P,\mu\right)\left(t^I,t^J\right)\\
\begin{aligned}
&=\left(\frac{\sqrt{-1}}{2\pi}\right)^n\int_{V_n}\frac{w^I\ol{w}^J}{\left(\Ht{w}Dw\right)^m}\cdot\frac{dw_0\wedge d\ol{w}_0\wedge\ldots\wedge dw_{n-1}\wedge d\ol{w}_{n-1}}{\left(\Ht{w}Dw\right)^{n+1}}\\
&=\left(\frac{\sqrt{-1}}{2\pi}\right)^n\int_{V_n}\frac{\prod\limits_{j=0}^{n-1}w_j^{I_j}\ol{w}_j^{J_j}\cdot
dw_0\wedge d\ol{w}_0\wedge\ldots\wedge dw_{n-1}\wedge
d\ol{w}_{n-1}}{\left(\mu_n\left(P\right)+\sum\limits_{j=0}^{n-1}\mu_j\left(P\right)\left|w_j\right|^2\right)^{m+n+1}}.
\end{aligned}
\end{multline*}

Using polar coordinates,
\begin{multline*}
Hilb_m\left(\varphi_P,\mu\right)\left(t^I,t^J\right)\\
\begin{aligned}
&=\frac{1}{\pi^n}\int_{\R_+^n}\frac{\prod\limits_{j=0}^{n-1}r_j^{I_j+J_j+1}dr_0\cdots dr_{n-1}}{\left(\mu_n\left(P\right)+\sum\limits_{j=0}^{n-1}\mu_j\left(P\right)r_j^2\right)^{m+n+1}}\prod_{j=0}^{n-1}\int_0^{2\pi}e^{i\left(I_j-J_j\right)\theta_j}d\theta_j\\
&=2^n\delta_{IJ}\int_{\R_+^n}\frac{\prod\limits_{j=0}^{n-1}r_j^{2I_j+1}}{\left(\mu_n\left(P\right)+\sum\limits_{j=0}^{n-1}\mu_j\left(P\right)r_j^2\right)^{m+n+1}}\cdot
dr_0\cdots dr_{n-1}.
\end{aligned}
\end{multline*}
In particular, $t^I$ and $t^J$ are orthogonal if $I\neq J$. Notice
that $t^I$ is monic with $\ell(t^I)=z^I=z_0^{I_0}\cdots
z_{n-1}^{I_{n-1}}$. These two facts imply, by Proposition
\ref{orthoChebysec}, that $t^I$ is the Chebyshev section
$Ch_{m,\frac{1}{m}(I_0,\ldots,I_{n-1})}$, or equivalently, for
$\alpha\in\frac{1}{m}\Delta_m(L)$,
$$
Ch_{m,\alpha}=t_0^{m\alpha_0}\cdots t_{n-1}^{m\alpha_{n-1}}\cdot
t_n^{m\left(1-\alpha_0-\cdots-\alpha_{n-1}\right)}.
$$
\end{proof}

For $x>0$, let
$$
\Gamma\left(x\right):=\int_0^{+\infty}e^{-t}t^{x-1}dt.
$$
and
$$
B\left(x_0,\ldots,x_n\right):=\frac{\Gamma\left(x_0\right)\cdots\Gamma\left(x_n\right)}{\Gamma\left(x_0+\cdots+x_n\right)}.
$$
The following integral formula for the multivariate beta function is
also needed
\cite[Equation 49.5]{KBJ}.

\begin{lem}\label{multibeta}
Let $k_0,\ldots,k_{n-1}\in\R_{>0}$ and $\displaystyle
l>\frac{1}{2}\sum_{i=0}^{n-1}k_i$. Then
$$
2^n\int_{\R_+^n}\frac{\prod\limits_{j=0}^{n-1}x_j^{k_j-1}dx_0\cdots
dx_{n-1}}{\left(1+\sum\limits_{j=0}^{n-1}x_j^2\right)^l}=B\left(\frac{k_0}{2},\ldots,\frac{k_{n-1}}{2},l-\frac{1}{2}\sum_{i=0}^{n-1}k_i\right).
$$
\end{lem}

\begin{ptn}
The norms of the section $t^I$ is given by
$$
\left\|t^I\right\|_{\varphi_P,\mu}^2=\frac{I!}{\left(m+n\right)!\prod\limits_{j=0}^n\mu_j\left(P\right)^{I_j+1}}.
$$
\end{ptn}
\begin{proof}
To compute the integral, one can use variables
$\rho_j:=\sqrt{\frac{\mu_j(P)}{\mu_n(P)}}r_j$ and apply Lemma
\ref{multibeta},
$$\begin{aligned}
\left\|t^I\right\|_{\varphi_P,\mu}^2&=2^n\int_{\R_+^n}\frac{\prod\limits_{j=0}^{n-1}r_j^{2I_j+1}\cdot dr_0\cdots dr_{n-1}}{\mu_n\left(P\right)^{m+n+1}\left(1+\sum\limits_{j=0}^{n-1}\frac{\mu_j\left(P\right)}{\mu_n\left(P\right)}r_j^2\right)^{m+n+1}}\\
&=2^n\int_{\R_+^n}\frac{\prod\limits_{j=0}^{n-1}\frac{\mu_n\left(P\right)^{I_j+1}}{\mu_j\left(P\right)^{I_j+1}}\prod\limits_{j=0}^{n-1}\rho_j^{2I_j+1}\cdot d\rho_0\cdots d\rho_{n-1}}{\mu_n\left(P\right)^{m+n+1}\left(1+\sum\limits_{j=0}^{n-1}\rho_j^2\right)^{m+n+1}}\\
&=\frac{2^n}{\prod\limits_{j=0}^n\mu_j\left(P\right)^{I_j+1}}\int_{\R_+^n}\frac{\prod\limits_{j=0}^{n-1}\rho_j^{2I_j+1}\cdot d\rho_0\cdots d\rho_{n-1}}{\left(1+\sum\limits_{j=0}^{n-1}\rho_j^2\right)^{m+n+1}}\\
&=\frac{1}{\prod\limits_{j=0}^n\mu_j\left(P\right)^{I_j+1}}B\left(I_0+1,\ldots,I_n+1\right)\\
&=\frac{I!}{\left(m+n\right)!\prod\limits_{j=0}^n\mu_j\left(P\right)^{I_j+1}}.
\end{aligned}$$
\end{proof}

Recall Stirling's formula
\begin{lem}\label{stl}
$$
\log n!=n\log n-n+\log\sqrt{2\pi n}+o\left(1\right)
$$
as $n\to\infty$.
\end{lem}
\begin{cor}\label{stlquo}
Let $m$ be a positive integer. Then
$$
\frac{1}{k}\log\left(km\right)!=m\log km-m+o\left(1\right)
$$
as $k\to\infty$.
\end{cor}
\begin{proof}
By Lemma \ref{stl},
$$\begin{aligned}
\frac{1}{k}\log\left(km\right)!&=m\log km-m+\frac{1}{k}\log\sqrt{2\pi km}+o\left(\frac{1}{k}\right)\\
&=m\log km-m+o\left(1\right)
\end{aligned}$$
as $k\to\infty$.
\end{proof}
\begin{cor}\label{stlsum}
Let $m_0,\ldots,m_n$ be positive integers and $M=m_0+\cdots+m_n$.
Then
$$
\lim_{k\to\infty}\frac{1}{k}\log\frac{\prod\limits_{j=0}^n\left(km_j\right)!}{\left(kM\right)!}=\sum_{j=0}^nm_j\log\frac{m_j}{M}.
$$
\end{cor}
\begin{proof}
By Corollary \ref{stlquo},
$$\begin{aligned}
\lim_{k\to\infty}\frac{1}{k}\log\frac{\prod\limits_{j=0}^n\left(km_j\right)!}{\left(kM\right)!}&=\lim_{k\to\infty}\left(\sum_{j=0}^n\left(m_j\log km_j-m_j\right)-\sum_{j=0}^nm_j\log kM+M\right)\\
&=\sum_{j=0}^nm_j\log\frac{m_j}{M}.
\end{aligned}$$
\end{proof}

\begin{proof}[Proof of Theorem \ref{CompCheby}]
For $\alpha\in\Int\Delta(H)\cap\Q^n$, one can find $M\in\N$ so that
$M\alpha\in\N^n$. In Definition \ref{Chebydef} pick $m_k=Mk$ and
$\alpha_{m_k}=\alpha$. By Corollary \ref{stlsum},
\begin{equation}\label{cdense}\begin{aligned}
c\left[\varphi_P\right]\left(\alpha\right)&=\lim_{k\to\infty}\frac{1}{Mk}\log\left\|Ch_{Mk,\alpha}\right\|_{\varphi_P,\mu}^2\\
&=\lim_{k\to\infty}\frac{1}{Mk}\log\frac{\left(Mk\left(1-\sum\limits_{j=0}^{n-1}\alpha_j\right)\right)!\prod\limits_{j=0}^{n-1}\left(Mk\alpha_j\right)!}{\left(Mk+n\right)!\mu_n\left(P\right)^{Mk\left(1-\sum\limits_{j=0}^{n-1}\alpha_j\right)+1}\prod\limits_{j=0}^{n-1}\mu_j\left(P\right)^{Mk\alpha_j+1}}\\
&=-\left(1-\sum_{j=0}^{n-1}\alpha_j\right)\log\mu_n\left(P\right)-\sum_{j=0}^{n-1}\alpha_j\log\mu_j\left(P\right)\\
&\qquad\qquad\qquad+\frac{1}{M}\lim_{k\to\infty}\frac{1}{k}\log\frac{\left(kM\left(1-\sum\limits_{j=0}^{n-1}\alpha_j\right)\right)!\prod\limits_{j=0}^{n-1}\left(kM\alpha_j\right)!}{\left(kM\right)!}\\
&=-\left(1-\sum_{j=0}^{n-1}\alpha_j\right)\log\mu_n\left(P\right)-\sum_{j=0}^{n-1}\alpha_j\log\mu_j\left(P\right)\\
&\qquad\qquad\qquad+\left(1-\sum_{j=0}^{n-1}\alpha_j\right)\log\left(1-\sum_{j=0}^{n-1}\alpha_j\right)+\sum_{j=0}^{n-1}\alpha_j\log\alpha_j\\
&=\left(1-\sum_{j=0}^{n-1}\alpha_j\right)\log\frac{\left(1-\sum\limits_{j=0}^{n-1}\alpha_j\right)}{\mu_n\left(P\right)}+\sum_{j=0}^{n-1}\alpha_j\log\frac{\alpha_j}{\mu_j\left(P\right)}.
\end{aligned}\end{equation}
By Proposition \ref{cvx}, $c[\varphi_P]$ is convex and hence
continuous on $\Int\Delta(H)$. Therefore (\ref{cdense}) holds for
all $\alpha\in\Int\Delta(H)$. This conludes the proof of Theorem
\ref{CompCheby}.
\end{proof}

\begin{rmk}\label{cob}
The notation $\varphi_P$ depends on the choice of the coordinates
$Z$, as is implied by (\ref{cov3}) and (\ref{cov4}). Indeed, a
section $s$ of $mH$ can be expressed as a polynomial $f(Z)$ of
degree $m$. By definition, its norm at a point $p=[Z_0:\cdots:Z_n]$
is given by
$$
\left\|s\left(p\right)\right\|^2=\frac{\left|f\left(Z\right)\right|^2}{\left(\Ht{Z}PZ\right)^m}.
$$
Note that the right hand side is well-defined since both the
numerator and denominator are homogeneous of degree $2m$. Using a
new choice of coordinates $W$ given by $Z=AW$, the same section $s$
is expressed as the polynomial $g(W):=f(AW)$ in $W$, and its norm at
the same point $p$ is expressed as
$$
\left\|s\left(p\right)\right\|^2=\frac{\left|f\left(Z\right)\right|^2}{\left(\Ht{Z}PZ\right)^m}=\frac{\left|g\left(W\right)\right|^2}{\left(\Ht{W}\Ht{A}PAW\right)^m}.
$$
Therefore, after switching to the new coordinates, the same metric
has potential $\varphi_{\Ht{A}PA}$.
\end{rmk}

This will be used in the proof of Proposition \ref{geoddecomplem}.

\section{Geodesics of Fubini--Study potentials}\label{Secgeod}

This section introduces Bergman geodesics and explains how to use
them to solve boundary value problems of the geodesic equation. As
an application, all geodesics in the spaces of Fubini--Study metrics
on $\PP^n$ are classified.

First recall the definition of a geodesic \cite[Proposition 3]{Don}.
\begin{dfn}
Let $(X^n,\omega)$ be a compact K\"ahler manifold. Define $S$ to be
the unit strip $\{\re s\in(0,1)\}$ and $\pi_2$ the projection map
$S\times X\to X$. A function $u:S\times X\to\R$ is called a geodesic
if $u(s,\cdot)$ is $\omega$-plurisubharmonic, is independent of $\im
s$, and
$$
\left(\pi_2^*\omega+\sqrt{-1}\partial\ol{\partial}u\right)^{n+1}=0.
$$
\end{dfn}

The boundary value problem of the geodesic equation has
\cite[Theorem 1.3]{Blo} a unique \cite[Theorem 3, Corollary 2]{Che}
continuous solution.
\begin{ptn}
For smooth boundary conditions $u_0$ and $u_1$, there is a unique
continuous geodesic $u$ connecting them, i.e., $u(i,\cdot)=u_i$,
$i=0,1$.
\end{ptn}

To obtain the solution, one approach is to use quantization. The
following is proven by Phong--Sturm \cite[Theorem 1]{PS} and
Berndtsson \cite[Theorem 6.1]{Ber}.
\begin{dfn}\label{Bgdef}
Let $X^n$ be a projective K\"ahler manifold with $L\to X$ an ample
line bundle. Given metrics with potentials $\varphi(0,\,\cdot\,)$
and $\varphi(1,\,\cdot\,)$ on $L$ with positive curvature, one can
equip $H^0(X,mL)$ with Hermitian metrics $H_{m,0}$ and $H_{m,1}$ as
follows,
\begin{equation}\label{Hmdef}
H_{m,j}\left(s,t\right)=\int_Xs\ol{t}e^{-m\varphi\left(j,\,\cdot\,\right)}\frac{\left(\sqrt{-1}\partial\ol{\partial}\varphi\left(j,\,\cdot\,\right)\right)^n}{n!},
\end{equation}
for $j=0,1$. Using simultaneous diagonalization, one can find
orthonormal bases $\{s_{m,j}\}$ and $\{e^{\lambda_j}s_{m,j}\}$ for
$H_{m,0}$ and $H_{m,1}$, respectively. Define the \textit{Bergman
geodesic}
\begin{equation}\label{Berggeo}
\varphi_m\left(t,\,\cdot\,\right):=\frac{1}{m}\log\sum_je^{\lambda_jt}\left|s_{m,j}\right|^2.
\end{equation}
\end{dfn}

\begin{thm}\label{Bg}
Notation as in Definition \ref{Bgdef}. The continuous geodesic
$\varphi$ connecting $\varphi(0,\,\cdot\,)$ and
$\varphi(1,\,\cdot\,)$ is given by
$$
\varphi=\lim_{m\to\infty}\varphi_m.
$$
\end{thm}

\begin{lem}\label{geodexm}
Notation as in Theorem \ref{CompCheby}. For a real diagonal matrix
$D=\diag\{d_0,\ldots,d_n\}$, $\varphi_{e^{tD}}$ is a geodesic.
\end{lem}
\begin{proof}
It suffices to compute the geodesic connecting
$\varphi_I=\log\Ht{z}z$ and $\varphi_{e^D}=\log\Ht{z}e^Dz$.

First note that the linear space $H^0(\PP^n,mH)$ is spanned by
monomials in $s_0=z_0,\ldots,s_{n-1}=z_{n-1},s_n=1$ of degree $m$.
By (\ref{Hmdef}), for multi-indices $I$ and $J$ with $|I|=|J|=m$,
$$\begin{aligned}
H_{m,0}\left(s^I,s^J\right)&=\int_X\frac{z^I\bar{z}^J}{\left\|z\right\|^{2m}}\cdot\frac{\sqrt{-1}^ndz_0\wedge d\ol{z}_0\wedge\ldots\wedge dz_{n-1}\wedge d\ol{z}_{n-1}}{\left\|z\right\|^{2n+2}}\\
&=2^n\int_{\R_+^n}\frac{r^{I+J+1}}{\left(1+\sum\limits_{i=0}^{n-1}r_i^2\right)^{n+m+1}}dr\int_{[0,2\pi)^n}e^{i\left(I-J\right)\theta}d\theta\\
&=\left(2\pi\right)^n\delta_{IJ}B\left(I_0+1,\ldots,I_n+1\right).
\end{aligned}$$
Similarly, applying a change of variable $W_j=e^\frac{d_j}{2}Z_j$
one has
$$
H_{m,1}\left(\prod_j\left(e^\frac{d_j}{2}s_j\right)^{I_j},\prod_j\left(e^\frac{d_j}{2}s_j\right)^{J_j}\right)=\left(2\pi\right)^n\delta_{IJ}B\left(I_0+1,\ldots,I_n+1\right).
$$
Define
$$
s_{m,I}:=\frac{z^I}{\sqrt{\left(2\pi\right)^nB\left(I_0+1,\ldots,I_n+1\right)}}.
$$
Then $\{s_{m,I}\}$ and
$\{e^{\frac{1}{2}\sum\limits_jI_jd_j}s_{m,I}\}$ are orthonormal
bases for $H_{m,0}$ and $H_{m,1}$, respectively. By (\ref{Berggeo}),
the Bergman geodesic
$$\begin{aligned}
\varphi_m\left(t,z\right)&=\frac{1}{m}\log\sum_{\left|I\right|=m}e^{\sum\limits_jI_jd_jt}\cdot\frac{\left|z\right|^{2I}}{\left(2\pi\right)^nB\left(I_0+1,\ldots,I_n+1\right)}\\
&=\frac{1}{m}\log\sum_{\left|I\right|=m}\frac{\prod\limits_j\left(e^{d_jt}\left|z_j\right|^2\right)^{I_j}}{B\left(I_0+1,\ldots,I_n+1\right)}-\frac{n}{m}\log2\pi\\
&=\frac{1}{m}\log\sum_{\left|I\right|=m}\frac{m!}{I!}\prod_j\left(e^{d_jt}\left|z_j\right|^2\right)^{I_j}+\frac{1}{m}\log\frac{\left(n+m\right)!}{m!}-\frac{n}{m}\log2\pi\\
&=\frac{1}{m}\log\left(\sum_je^{d_jt}\left|z_j\right|^2\right)^m+\frac{1}{m}\log\frac{\left(n+m\right)!}{m!}-\frac{n}{m}\log2\pi\\
&=\varphi_{e^{tD}}\left(z\right)+\frac{1}{m}\log\frac{\left(n+m\right)!}{m!}-\frac{n}{m}\log2\pi.
\end{aligned}$$
Using Theorem \ref{Bg}, the geodesic
$$
\varphi=\lim_{m\to\infty}\varphi_m\left(t,\,\cdot,\,\right)=\varphi_{e^{tD}}.
$$
\end{proof}


\begin{ptn}\label{geoddecomplem}
For matrices $A\in GL(n+1,\C)$ and $D$ real diagonal,
$\varphi_{P(t)}$ is a geodesic, where
$$
P\left(t\right):=\Ht{A}e^{tD}A.
$$
\end{ptn}
\begin{proof}
This follows from Lemma \ref{geodexm} and Remark \ref{cob} by a
change of variable $W=AZ$.
\end{proof}

\begin{proof}[Proof of Proposition \ref{introgeo}]
``$\Rightarrow$'': Let $\varphi_{P(t)}$ be a geodesic. By
simultaneously diagonalizing $P(0)$ and $P(1)$, one can write
$P(0)=\Ht{A}A$ and $P(1)=\Ht{A}e^DA$ for some $A\in GL(n+1,\C)$ and
$D$ real diagonal. Define
$$
\wt{P}\left(t\right):=\Ht{A}e^{tD}A.
$$
By Proposition \ref{geoddecomplem}, $\varphi_{\wt{P}(t)}$ is a
geodesic connecting $P(0)$ and $P(1)$. Hence by uniqueness
$\varphi_{\wt{P}(t)}=\varphi_{P(t)}$.

``$\Leftarrow$'': This is Proposition \ref{geoddecomplem}.
\end{proof}

\section{A counterexample}\label{Seccount}

This section proves Corollary \ref{introaff}.

\begin{lem}\label{linedecomp}
For a geodesic $\varphi_{P(t)}$, the following are equivalent
\begin{enumerate}[1)]
    \item\label{muexp}$\mu_i(P(t))=e^{k_i t}$, $k_i\in\R$;
    \item\label{lunidecomp}There is a lower unitriangular matrix $L$ such that $P(t)=\Ht{L}e^{Kt}L$, where $K=\diag\{k_0,\ldots,k_n\}$, $k_i\in\R$.
\end{enumerate}
Moreover, the decomposition $P(t)=\Ht{L}e^{Kt}L$ is unique.
\end{lem}
\begin{proof}
\ref{muexp})$\Rightarrow$\ref{lunidecomp}): One proceeds by
induction on the size of $P$. When $P$ has order 1, the matrix $L$
is 1.

Suppose the statement is true for matrices of order $n$. Let $P(t)$
be of order $n+1$ with $\mu_i(P(t))=e^{k_i t}$. By Proposition
\ref{introgeo}, $P(t)=\Ht{A}e^{tD}A$. Define
$$
I:=\{i\,|\,d_i=k_n\}.
$$
By Definition \ref{defetamu}, for any $P\in\mathcal{P}_{n+1}(\C)$,
$\mu_n(P)=p_{nn}$. Thus
$$
\mu_n\left(P\left(t\right)\right)=\mu_n\left(\Ht{A}e^{tD}A\right)=\sum_{i=0}^n\left|a_{in}\right|^2e^{d_it}.
$$
Therefore,
$$
\sum_{i=0}^n\left|a_{in}\right|^2e^{d_it}=e^{k_nt}.
$$
Since this is true for all $t$, it follows that $a_{in}=0$ for
$i\notin I$, and
$$
\sum_{i\in I}\left|a_{in}\right|^2=1.
$$
In particular, $I\neq\varnothing$. By reordering the rows and
columns one may assume $I=\{m,m+1,\ldots,n\}$. This means the vector
$v_n=(a_{mn},\ldots,a_{nn})^\mathsf{T}\in\C^{n+1-m}$ has unit
length, so it can be completed to an orthonormal basis
$v_m,\ldots,v_n$. Define a unitary matrix
$$
U:=\begin{pmatrix}
I_m\\
&V
\end{pmatrix}
$$
where
$$
V:=\begin{pmatrix}
\Ht{v_m}\\
\vdots\\
\Ht{v_n}
\end{pmatrix}.
$$
Since $d_m=\cdots=d_n=k_n$, one has $e^{tD}=\Ht{U}e^{tD}U$, and
$$
P\left(t\right)=\Ht{A}e^{tD}A=\Ht{A}\Ht{U}e^{tD}UA.
$$

Write
$$
B:=UA.
$$
Recall that $a_{in}=0$ for $0\leq i<m$, and the vectors
$v_m,\ldots,v_n\in\C^{n+1-m}$ are orthonormal. Therefore for $0\leq
i<m$,
$$
b_{in}=a_{in}=0;
$$
for $m\leq i<n$,
$$
b_{in}=\Ht{v_i}v_n=0;
$$
and for $i<n$,
$$
b_{nn}=\Ht{v_n}v_n=1.
$$
That is,
$$
B=\begin{pmatrix}
\wt{B}&0\\
w&1
\end{pmatrix},
$$
where $\wt{B}\in\mathcal{P}_n(\C)$ and $w\in\C^n$. Write
$$
\wt{D}:=\diag\{d_0,\ldots,d_{n-1}\}.
$$
One then has the decomposition
$$\begin{aligned}
P\left(t\right)&=\begin{pmatrix}
\Ht{\wt{B}}&\Ht{w}\\
0&1
\end{pmatrix}\begin{pmatrix}
e^{\wt{D}t}\\
&e^{k_nt}
\end{pmatrix}\begin{pmatrix}
\wt{B}&0\\
w&1
\end{pmatrix}\\
&=\begin{pmatrix}
I_n&\Ht{w}\\
0&1
\end{pmatrix}\begin{pmatrix}
\Ht{\wt{B}}e^{\wt{D}t}\wt{B}\\
&e^{k_nt}
\end{pmatrix}\begin{pmatrix}
I_n&0\\
w&1
\end{pmatrix}.
\end{aligned}$$
By Claims \ref{mu} and \ref{mured}, for $0\leq i\leq n-1$,
$$
\mu_i\left(P\left(t\right)\right)=\mu_i\begin{pmatrix}
\Ht{\wt{B}}e^{\wt{D}t}\wt{B}\\
&e^{k_nt}
\end{pmatrix}=\mu_i\left(\Ht{\wt{B}}e^{\wt{D}t}\wt{B}\right).
$$
By the induction hypothesis, one can write
$\Ht{\wt{B}}e^{\wt{D}t}\wt{B}=\Ht{\wt{L}}e^{\wt{K}t}\wt{L}$, where
$\wt{L}$ is lower unitriangular and $\wt{K}$ is real diagonal. Then,
$$\begin{aligned}
P\left(t\right)&=\begin{pmatrix}
I_n&\Ht{w}\\
0&1
\end{pmatrix}\begin{pmatrix}
\Ht{\wt{B}}e^{\wt{D}t}\wt{B}\\
&e^{k_nt}
\end{pmatrix}\begin{pmatrix}
I_n&0\\
w&1
\end{pmatrix}\\
&=\begin{pmatrix}
I_n&\Ht{w}\\
0&1
\end{pmatrix}\begin{pmatrix}
\Ht{\wt{L}}e^{\wt{K}t}\wt{L}\\
&e^{k_nt}
\end{pmatrix}\begin{pmatrix}
I_n&0\\
w&1
\end{pmatrix}\\
&=\begin{pmatrix}
\Ht{\wt{L}}&\Ht{w}\\
0&1
\end{pmatrix}\begin{pmatrix}
e^{\wt{K}t}\\
&e^{k_nt}
\end{pmatrix}\begin{pmatrix}
\wt{L}&0\\
w&1
\end{pmatrix}\\
&=\Ht{L}e^{Kt}L,
\end{aligned}$$
where
$$
L=\begin{pmatrix}
\wt{L}&0\\
w&1
\end{pmatrix}
$$
is lower unitriangular and $K=\diag\{\wt{K},k_n\}$ is real diagonal.

\ref{lunidecomp})$\Rightarrow$\ref{muexp}): This follows directly
from Claim \ref{mu}.

For uniqueness, first notice that $K$ is determined by $P(t)$ under
the relation $\mu_i(P(t))=e^{k_i t}$. Suppose
$\Ht{L}e^{Kt}L=\Ht{L'}e^{Kt}L'$ for some lower unitriangular
matrices $L$ and $L'$. That is,
$L'L^{-1}=e^{-Kt}\Ht{LL'^{-1}}e^{Kt}$. Since $L'L^{-1}$ is lower
unitriangular and $e^{-Kt}\Ht{LL'^{-1}}e^{Kt}$ is upper triangular,
they are equal only when they both are identity matrices, which
means $L=L'$.
\end{proof}

\begin{ptn}\label{affdecomp}
For a geodesic $\varphi_{P(t)}$, the following are equivalent
\begin{enumerate}[1)]
    \item$\mu_i(P(t))=a_ie^{k_i t}$, where $a_i>0$ and $k_i\in\R$;
    \item There is a lower triangular matrix $L$ such that $P(t)=\Ht{L}e^{Kt}L$, where $l_{ii}=\sqrt{a_i}$ and $K=\diag\{k_0,\ldots,k_n\}$, $k_i\in\R$.
\end{enumerate}
Moreover, such a decomposition is unique.
\end{ptn}
\begin{proof}
Let $P'=\Ht{E}PE$, where
$E=\diag\{\frac{1}{\sqrt{a_0}},\ldots,\frac{1}{\sqrt{a_n}}\}$, and
apply Lemma \ref{linedecomp} to $P'$.
\end{proof}

\begin{proof}[Proof of Corollary \ref{introaff}]
By Theorem \ref{CompCheby}, the Chebyshev potential
$c[\varphi_{P(t)}]$ is affine if and only if $\log\mu_i(P(t))$ is
affine for any $i$. Using Proposition \ref{affdecomp} we conclude.
\end{proof}


\subsection{Discussion}

According to Mabuchi \cite{Mab} and Bando--Mabuchi \cite{BM} the
space of K\"ahler--Einstein metrics is connected and totally
geodesic submanifold of the space of \K metrics, and, further, it is
isometric to a finite-dimensional symmetric space of the form
$\Aut(X)/\hbox{\rm Iso}(X,\omega)$ where $\omega$ is a
K\"ahler--Einstein form. The geodesics of such symmetric spaces are
completely understood. It is tempting to conjecture that some of our
results should extend to this more general setting. However, there
are several potential pitfalls. First, one would need to extend the
Bando--Mabuchi results to the case of \K {\it potentials} and not
just {\it forms} (i.e., potentials mod $\R$). This can probably be
done following ideas in Darvas--Rubinstein (see, e.g.,
\cite[\S5.2]{DR}). Second, and more crucially, one would need to
compute Chebyshev sections, even with respect to a conveniently
chosen basis of anticanonical sections (assuming, say, that $-K_X$
is very ample).

Finally, we point out that one of key observations in the proof of
Theorem \ref{CompCheby} is that (in our simple setting of $\PP^n$
with the standard flag) the Chebyshev sections do not change along
the geodesic if and only if the associated geodesic in the symmetric
space of positive Hermitian matrices is of the form
$$
P\left(t\right)=\Ht{L}D\left(t\right)L,
$$
where $L$ is a lower unitriangular matrix independent of $t$, and
$D(t)$ is diagonal with positive diagonal entries. This essentially
means that while the curve of metrics is not a curve of toric
metrics, they are all toric up to a lower unitriangular
transformation, which are precisely the transformations that
preserve the flag, and hence the Chebyshev sections as well (unit is
needed to preserve the monic leading term)\footnote {After
completing this work and informing Reboulet of our results, he
kindly informed us of an independent and beautiful disproof of
Conjecture \ref{conj} using completely different ideas \cite{Reb2}
(although without explicit counterexamples nor explicit computation
of the Chebyshev transform).}.

\medskip
{\sc University of Maryland}

{\tt cjin123@terpmail.umd.edu, yanir@alum.mit.edu}

\end{document}